\documentclass[a4paper,11pt,oneside, reqno]{amsart}

\usepackage{enumerate}
\usepackage[utf8]{inputenc}
\usepackage[T1]{fontenc}
\usepackage{amssymb}
\usepackage{amsfonts}
\usepackage{url}
\usepackage{hyperref,srcltx,xcolor}
\usepackage{booktabs,tabularx,array}
\usepackage{graphicx}
\usepackage[width=0.95\textwidth,font=footnotesize]{caption}

\usepackage{mathtools}
\mathtoolsset{showonlyrefs}

\newcommand{\field}[1]{\mathbb{{#1}}}
\newcommand{\CM}{\field{C}}

\newcommand{\dd}{\mathrm{d}}

\newtheorem{theo}{Theorem}
\newtheorem{lem}[theo]{Lemma}

\theoremstyle{definition}
\newtheorem{rk}[theo]{Remark}

\numberwithin{equation}{section}

\newenvironment{acknowledgements}{\subsection*{Acknowledgements}}{}

\allowdisplaybreaks

\author[S. R. Garcia]{Stephan Ramon Garcia}
\address[S. R. Garcia]{Department of Mathematics and Statistics\\
         Pomona College\\
         610 N. College Ave.\\
         Claremont, CA 91711\\
         USA}
\email{stephan.garcia@pomona.edu}
\urladdr{\url{https://stephangarcia.sites.pomona.edu/}}

\author[L. Grenié]{Loïc Grenié}
\address[L. Grenié]{Dipartimento di Ingegneria gestionale, dell'informazione e della produzione\\
         Università di Bergamo\\
         viale Marconi 5\\
         I-24044 Dal\-mi\-ne\\
         Italy}
\email{loic.grenie@gmail.com}

\author[G. Molteni]{Giuseppe Molteni}
\address[G. Molteni]{Dipartimento di Matematica\\
         Università di Milano\\
         via Saldini 50\\
         I-20133 Milano\\
         Italy}
\email{giuseppe.molteni1@unimi.it}
\urladdr{\url{https://sites.unimi.it/molteni/}}

\title[Explicit upper and lower bounds]
      {Explicit upper and lower bounds on certain functions related to $\zeta(s)$ for $s>1$}


\begin{document}
\begin{abstract}
We provide explicit upper and lower bounds for certain commonly occurring expressions that involve the
Riemann zeta function $\zeta(s)$ and its derivatives on $s>1$.  For example, this improves upon existing
bounds for $(s-1)\zeta(s)$ and its derivatives, and for $(\zeta'/\zeta)(s)$.
\end{abstract}

\maketitle


\section{Introduction}
Our aim in this paper is to provide novel explicit bounds for certain commonly occurring expressions that
involve the Riemann zeta function $\zeta(s)$ and its derivatives on $s>1$. For example, we improve upon
existing bounds for expressions such as $(s-1)\zeta(s)$ and its derivatives, and for $(\zeta'/\zeta)(s)$.
By ``explicit'' we mean that our bounds are in the form of inequalities with specified ranges of
applicability, as opposed to asymptotic expressions with hidden implied constants. This project arose as
an outgrowth of other endeavors that required such bounds and from our realization that explicit bounds
either did not exist in the literature or were not strong enough for our purposes.

This introduction conducts a brief literature review of each family of bounds before stating our
corresponding theorem, each of which improves upon the bounds we found in the literature. The proofs of
our theorems are located in later sections so that we can present each result in context within the
literature. This provides a better flow and permits this paper to serve as a ``one-stop shop'' for
explicit bounds for functions related to $\zeta(s)$ on $s>1$.

\subsection{Bounds on \texorpdfstring{$(s-1)\zeta(s)$}{(s-1)zeta(s)}}
Since $\zeta(s)$ has a simple pole at $s=1$ with residue $1$, it follows that $\lim_{s\to1^+}
(s-1)\zeta(s)=1$.  It is therefore natural to seek explicit bounds describing this behavior.
The integral
test furnishes a first step in this direction, since it implies that
\begin{equation}\label{eq:IntegralTest}
 1 < (s-1)\zeta(s) < s
\end{equation}
for $s > 1$. In fact, $(s-1)\zeta(s)$ is strictly increasing there; see Subsection \ref{Subsection:s1zsprime}.
The convexity of $x^{-s}$ and some algebra offer the more precise bounds
\begin{equation}\label{eq:SLB2}
 \frac{1}{2}(s+1) < (s-1)\zeta(s) < (s-1) + \bigg(\frac{2}{3} \bigg)^{s-1}.
\end{equation}
In 2002, Bastien and Rogalski showed that $(s-1)\zeta(s) < 2^{s-1}$ for $s > 1$
\cite[Cor. 1, p. 918]{BastienRogalski}. Ramar\'{e} improved this upper bound to
\begin{equation}\label{eq:Ramare}
 (s-1)\zeta(s) \leq e^{\gamma(s-1)}
\end{equation}
for $s>1$ \cite[Lem. 5.4]{Ramare:Density}. Here $\gamma$ denotes the Euler--Mascheroni constant.
The upper bound in \eqref{eq:SLB2} beats \eqref{eq:Ramare} for $s > 1.187$. 
In the opposite direction, Van de Lune showed that
\begin{equation}\label{eq:Lune}
 1 + \gamma(s-1) < (s-1)\zeta(s)
\end{equation}
for $s> 1$ \cite[Cor.~2.2]{vandeLune}.
This improves the lower bound in \eqref{eq:SLB2} for every $s$ and provides a good approximation 
when $s$ is close to $1$, since that linear polynomial is the first order Taylor expansion of 
$(s-1)\zeta(s)$ at $s=1$.

For each $N\geq 1$, applying the integral test to the tail yields
\begin{equation}\label{eq:L0}
    (s-1)\zeta(s) >  (s-1)\sum_{n=1}^{N-1}\frac{1}{n^s} + \frac{1}{N^{s-1}}.
\end{equation}
Denote the right side by $L_{0,N}(s)$. It has only zeroth-order contact with $(s-1)\zeta(s)$ at 
$s=1$, whereas Van de Lune's bound has first-order contact.  Away from $s=1$, however, $L_{0,N}$ 
is more effective, and increasing $N$ improves the bound.

Our first result improves upon the bounds above on $(s-1)\zeta(s)$.
Below $H_n=\sum_{k=1}^n 1/k$ denotes the $n$th harmonic number, with the convention that $H_0 = 0$,
and $\gamma_1= -0.0728\ldots$ denotes the first Stieltjes constant; see Subsection \ref{Subsection:Stieltjes} 
for the definition.
\begin{theo}\label{Theorem:JustZeta}
For $s>1$ and any $N\geq 1$,
\begin{equation*}
 \max\big\{ L_{0,N}(s),L_2(s) \big\}
 < (s-1)\zeta(s) 
 < \min\big\{ U_{1,N}(s),U_2(s) \big\},
\end{equation*}
where
\begin{align*}
L_{0,N}(s) 
&= (s-1)\sum_{n=1}^{N-1}\frac{1}{n^s} + \frac{N}{N^s}, \\
L_2(s) 
&= s + \frac{(1-\gamma)^2}{2\gamma_1}\ln\Big(1 - \frac{2\gamma_1}{1-\gamma}(s-1)\Big), \\
U_{1,N}(s) 
&= (s-1)\sum_{n=1}^{N-1}\frac{1}{n^s}+e^{(\gamma-H_{N-1})(s-1)}, \\
U_2(s) 
&= s - \sqrt{8}(1-\gamma)\arctan\bigg(\frac{s-1}{\sqrt{8}}\bigg) - 8\gamma_1 \ln\bigg(1 + \frac{(s-1)^2}{8}\bigg).
\end{align*}
\end{theo}

The subscript indicates the order of contact at $s=1$; that is, $L_{0,N}$ agrees in value, $U_{1,N}$ 
agrees through the linear term, and $L_2$ and $U_2$ agree through the quadratic term.
Note that $U_{1,1}(s)$ is Ramar\'{e}'s bound \eqref{eq:Ramare}. 

The bound $L_{0,N}(s)$ has been established in \eqref{eq:L0}.
The bound $U_{1,N}$ is established in Section \ref{Section:JustZeta}.  
The bounds $L_2$ and $U_2$ follow from Theorem \ref{Theorem:FirstDerivative} and
are obtained at the end of Section \ref{Section:[(s-1)zeta(s)]'}.

Below you can find where the strongest bounds on $(s-1)\zeta(s)$ lie when the
new $N$-dependent families are restricted to $N=2$.
\begin{center}
\small
\renewcommand{\arraystretch}{1.25}
\begin{tabularx}{\textwidth}{@{}
>{\raggedright\arraybackslash}p{0.16\textwidth}
>{\raggedright\arraybackslash}p{0.22\textwidth}
>{\raggedright\arraybackslash}p{0.16\textwidth}
>{\raggedright\arraybackslash}p{0.19\textwidth}
>{\raggedright\arraybackslash}X@{}}
\toprule
Range of $s$
& Best cited lower bound
& Best new lower bound
& Best cited upper bound
& Best new upper bound \\
\midrule
$1<s<s_R$
& Van de Lune \eqref{eq:Lune}
& $L_2(s)$
& Ramar\'e \eqref{eq:Ramare}
& $U_2(s)$ \\

$s_R<s<s_U$
& Van de Lune \eqref{eq:Lune}
& $L_2(s)$
& convexity \eqref{eq:SLB2}
& $U_2(s)$ \\

$s_U<s<s_L$
& Van de Lune \eqref{eq:Lune}
& $L_2(s)$
& convexity \eqref{eq:SLB2}
& $U_{1,2}(s)$ \\

$s>s_L$
& Van de Lune \eqref{eq:Lune}
& $L_{0,2}(s)$
& convexity \eqref{eq:SLB2}
& $U_{1,2}(s)$ \\
\bottomrule
\end{tabularx}
\end{center}
Here $s_R=1.187056\ldots$ denotes the crossover between the convexity bound \eqref{eq:SLB2} and 
Ramar\'{e}'s bound \eqref{eq:Ramare};
$s_U=1.634299\ldots$ is the crossover between $U_2$ and $U_{1,2}$; and
$s_L=4.032266\ldots$ is the crossover between $L_2$ and $L_{0,2}$.
%
%

\subsection{Bounds on \texorpdfstring{$[(s-1)\zeta(s)]'$}{[(s-1)zeta(s)]'}}\label{Subsection:s1zsprime}
Lower bounds on $[(s-1)\zeta(s)]'$ typically appear as part of greater endeavors, and upper bounds
seem little researched. For $s>1$, Delange proved that
\begin{equation}\label{eq:DelangeDerivative}
 \frac{1}{2s} < [(s-1)\zeta(s)]'
\end{equation}
when proving his lower bound \eqref{eq:DelangeZ'Z} on $(\zeta'/\zeta)(s)$ \cite[p. 334]{Delange:Remarque}.
An improved lower bound comes from Alzer and Kwong \cite{AlzerKwong}, who showed that
\begin{equation}\label{eq:AlzerKwong}
 \frac{\zeta'}{\zeta}(s) + \frac{1}{s-1} > \frac{\gamma}{s}
\end{equation}
for $s>1$, which ensures that
\begin{equation}\label{eq:ImplicitAlzer}
 [(s-1)\zeta(s)]' > \frac{\gamma}{s} (s-1)\zeta(s) > \frac{\gamma}{s}
\end{equation}
by \eqref{eq:IntegralTest}, which improves Delange's inequality \eqref{eq:DelangeDerivative} since
$\gamma > \frac{1}{2}$.

Our next result provides explicit upper and lower bounds on $[(s-1)\zeta(s)]'$. These bounds agree with
$[(s-1)\zeta(s)]'$ to the first order at $s=1$ (equalities of value and first derivative) and to zeroth
order at $s\to\infty$.
Consequently, they improve upon the lower bounds \eqref{eq:DelangeDerivative} and \eqref{eq:ImplicitAlzer}.
As mentioned above, upper bounds appear to be absent in the literature.
\begin{theo}\label{Theorem:FirstDerivative}
For $s>1$,
\begin{equation*}
1 - \frac{(1-\gamma)^2}{1-\gamma - 2\gamma_1(s-1)}
< [(s-1)\zeta(s)]'
< 1- \frac{1-\gamma + 2\gamma_1(s-1)}{1 + \tfrac{1}{8}(s-1)^2}.
\end{equation*}
\end{theo}
The proof of this theorem is in Section \ref{Section:[(s-1)zeta(s)]'}. This is presented somewhat out of
order, since it involves the integration of the bound that appears in Theorem \ref{Theorem:DoublePrime} below.
We have elected to present things in the present order since it seems natural for the first derivative to
come before the second.

\subsection{Bounds on \texorpdfstring{$[(s-1)\zeta(s)]''$}{[(s-1)zeta(s)]''}}
Our search of the literature produced no published bounds on the second derivative of $(s-1)\zeta(s)$,
although that does not mean they are not of current or future interest. For example, we require such
second-derivative bounds to prove Theorem \ref{Theorem:FirstDerivative} above.
\begin{theo}\label{Theorem:DoublePrime}
For $1 < s\leq 4$,
\begin{equation*}
\scalebox{1.25}{$
\frac{-2\gamma_1(1-\gamma)^2}{(1-\gamma - 2\gamma_1(s-1))^2}
$}
<
[(s-1)\zeta(s)]''
<
\scalebox{1.25}{$
-\frac{2\gamma_1-\frac{1}{4}(1 - \gamma)(s-1) - \frac{1}{4}\gamma_1 (s-1)^2}{(1 + \frac{1}{8}(s-1)^2)^2}.
$}
\end{equation*}
\end{theo}
The proof of the previous result is in Section \ref{Section:DoublePrime}.
One could potentially improve upon the bounds above, which agree with $[(s-1)\zeta(s)]''$ to the zeroth
order at $s=1$. One could also pursue
bounds with a greater range of applicability. However, the small $s>1$ regime is of primary importance
and our main application of Theorem \ref{Theorem:DoublePrime} is to the proof of Theorem
\ref{Theorem:FirstDerivative}, so we do not pursue such improvements here.
%
%

\subsection{Bounds on \texorpdfstring{$(\zeta'/\zeta)(s)$}{zeta'(s)/zeta(s)}}
In the classical proofs of zero-free regions for the zeta function,
one has to rely on a lower bound on $(\zeta'/\zeta)(s)+(s-1)^{-1}$ for $s>1$. For example,
in \cite[Lem. 70.1]{Hall-Tenenbaum:Divisors} Hall and Tenenbaum prove that
\begin{equation*}
 \frac{\zeta'}{\zeta}(s)+\frac{1}{s-1} > 0.
\end{equation*}
A significant improvement, established by Delange \cite{Delange:Remarque} in 1987, is
\begin{equation}\label{eq:DelangeZ'Z}
 \frac{\zeta'}{\zeta}(s)+\frac{1}{s-1} > \frac{1}{2s^2}.
\end{equation}
This was improved in 2021 by Alzer and Kwong \cite[Thm. 1.1]{AlzerKwong}, who obtained the lower bound
$\gamma/s$ in \eqref{eq:AlzerKwong}. Hilberdink increased this lower bound to
\begin{equation}\label{eq:Hilberdink}
 \frac{\zeta'}{\zeta}(s)+\frac{1}{s-1} > \frac{\gamma}{1+\gamma(s-1)}
\end{equation}
for $s>1$ in 2023 \cite[eq. (1.5)]{Hilberdink}. Then Leong proved in 2026 that
\begin{equation}\label{eq:Leong}
 \frac{\zeta'}{\zeta}(s) + \frac{1}{s-1} > \frac{1-\phi_1(s,k)}{s-1}
\end{equation}
for all $1<s<\sigma_k$, in which $k\geq 1$ is an integer and $\sigma_k\in(1,3)$; both $\sigma_k$
and $-\phi_1(s,k)$ are increasing in $k$ \cite{Leong:ExplicitEstimates}. The function $\phi_1(s,k)$
depends upon another function $\phi_0(s,k)$, so we refer the reader to \cite{Leong:ExplicitEstimates}
for further details.
Finally, we mention the inequality
\begin{equation}\label{eq:bbb}
\frac{\zeta'}{\zeta}(s) + \frac{\ln 2}{2^{s-1}-1} \geq 0,
\end{equation}
for $s > 1$, which follows by the fact that the modified zeta function $\eta(s) = (1-2^{1-s})\zeta(s)$ is
increasing since convex (see \cite{AdellLekuona2}) and $\eta'(1)=(\gamma-\frac{\ln 2}{2})\ln 2>0$. %
An alternative approach to \eqref{eq:bbb} is in \cite{AlexeevTao}, where the inequality is quickly deduced 
from a suitable probabilistic argument.

We prove the following result, whose proof is in Section \ref{Section:Combo}.
\begin{theo}\label{Theorem:Combo}
Let $c = 2\gamma - 2\gamma_1 - \gamma^2 = 0.9668\ldots$. For $s>1$,
\begin{equation*}
\frac{\gamma + c(s-1)}{1 + 2(s-1)+c(s-1)^2}
< \frac{\zeta'}{\zeta}(s) + \frac{1}{s-1}
< \frac{\gamma}{1+(\gamma+2\gamma_1/\gamma)(s-1)}.
\end{equation*}
\end{theo}
%
%
The lower bound of Theorem \ref{Theorem:Combo} improves upon Hilberdink's bound \eqref{eq:Hilberdink}
and Leong's bounds \eqref{eq:Leong} in the full range where they are defined. It improves also upon 
\eqref{eq:bbb} for $s \leq 3.4$.
One can take $\gamma$ itself as an upper bound.
For $s > 1 - \frac{\gamma}{2\gamma_1} = 4.9635\ldots$, the upper bound $(s-1)^{-1}$ is stronger than the
bound of Theorem \ref{Theorem:Combo}, so the advantage in our upper bound is in the (more important) small
$s>1$ regime.

Since
\begin{equation*}
 \lim_{s\to1^+} \bigg(\frac{\zeta'}{\zeta}(s) + \frac{1}{s-1} \bigg) = \gamma,
\end{equation*}
the next theorem, whose proof is in Section \ref{Section:Decreasing},
immediately yields the (weak) upper bound $\gamma$.
Nevertheless, it establishes something different since upper bounds alone cannot tell us whether a function
is decreasing or not.

\begin{theo}\label{Theorem:Decreasing}
$\displaystyle \frac{\zeta'}{\zeta}(s)+\frac{1}{s-1}$ is strictly decreasing on $s>1$.
\end{theo}

\subsection{Bounds on \texorpdfstring{$\zeta'(s)$}{zeta'(s)}}
The derivative of the zeta function on $s>1$ has been the focus of interest over the years, with a series
of authors obtaining explicit bounds. For example, Leong \cite[Lem. 11]{Leong:ExplicitEstimates} showed
that $\zeta'(s) + (s-1)^{-2} > -(s-3)^{-2}$ for $1 < s < 3$, but Van de Lune \cite[Theorem 2]{vandeLune}
already contains the bound $\zeta'(s) + (s-1)^{-2} > 0$ for every $s>1$, then recently reproved by
Hilberdink \cite{Hilberdink}. Alzer and Kwong obtained $\zeta'(s)+(s-1)^{-2}>\frac{3}{50}$ for $1 < s
\leq 1.1$ \cite[Lem. 2.1]{AlzerKwong}. Upper bounds in this context do not appear to be as well studied
as lower bounds.

Our next result bounds $\zeta'(s)+(s-1)^{-2}$ on $s>1$ from above and below by rational functions whose
Taylor series at $s=1$ begin with $-\gamma_1 + \gamma_2(s-1)$; the upper bound below matches to second order.
They hold for all $s>1$ and improve upon all cited lower bounds.
%
\begin{theo}\label{Theorem:ZetaPrime}
For $s>1$,
\begin{equation*}
\frac{\alpha}{ 1 + \beta (s-1) + \alpha (s-1)^2 }
< \zeta'(s) + \frac{1}{(s-1)^2}
< \frac{ \alpha}{1+ \beta(s-1) +\delta(s-1)^2},
\end{equation*}
in which 
$\alpha = -\gamma_1 = 0.0728\ldots$,
$\beta  = \frac{\gamma_2}{\gamma_1} = 0.1330\ldots$,
and
$\delta = (\frac{\gamma_2}{\gamma_1})^2 - \frac{\gamma_3}{2 \gamma_1} = 0.0318\ldots$.
\end{theo}
Simpler, but weaker, bounds that follow from the above are
\begin{equation*}
\frac{1}{14s^2}
< \zeta'(s) + \frac{1}{(s-1)^2}
< \frac{11}{151+(s-1)^2}.
\end{equation*}
The proof of Theorem \ref{Theorem:ZetaPrime} is in Section \ref{Section:ZetaPrime}.

\begin{acknowledgements}
LG and GM are members of the INdAM group GNSAGA. SRG is partially supported by NSF grant DMS-2452084. 
The initial draft of this article was submitted to the \href{https://arxiv.org/abs/2607.08342}{\texttt{arXiv}} 
on 9 July, 2026. Ensuing revisions utilized \texttt{ChatGPT} for clarity, proofreading, and adding some details.  
\end{acknowledgements}

\section{Preliminaries}\label{Section:Preliminaries}
This section contains results that are used elsewhere in this paper.
Subsection \ref{Subsection:Stieltjes} concerns the Stieltjes constants, which arise in the
Laurent expansion of $\zeta(s)$ around $s=1$.  In Subsection \ref{Subsection:Bernstein},
we explain how Bernstein polynomials can help certify the nonnegativity of
polynomials on an interval.

\subsection{Stieltjes constants}\label{Subsection:Stieltjes}
The Laurent-series coefficients for $\zeta(s)$ at $s=1$ appear in some of our theorem statements and
proofs, so we require a few words about them. Since $\zeta(s)$ has a simple pole at $s=1$ with residue
$1$, the Laurent series has principal part $(s-1)^{-1}$.
The \emph{Stieltjes constants} are implicitly defined by
\begin{equation}\label{eq:Stieltjes}
 \zeta(s) = \frac{1}{s-1} + \sum_{n=0}^{\infty}\frac{(-1)^n}{n!}\gamma_n(s-1)^n,
\end{equation}
in which the power series converges for all complex $s$. In particular, $\gamma_0 = \gamma = 0.5772\ldots$ 
is the Euler--Mascheroni constant; see Table \ref{Table:Stieltjes} for numerical values of the first 
several Stieltjes constants.
The values of $\gamma_n$ can be computed with \texttt{PARI}, \texttt{Mathematica}, or with the tools 
in \cite{AdellLeukona}. Two web resources are \cite{lmfdb,Plouffe}.

\begin{table}
\begin{equation*}
\begin{array}{c|l}
n & \phantom{-}\gamma_n\\
\hline
 0 & \phantom{-}0.577215664901532860606512090082\ldots \\
 1 &          - 0.072815845483676724860586375874\ldots \\
 2 &          - 0.009690363192872318484530386035\ldots \\
 3 & \phantom{-}0.002053834420303345866160046542\ldots \\
 4 & \phantom{-}0.002325370065467300057468170177\ldots \\
 5 & \phantom{-}0.000793323817301062701753334877\ldots \\
\end{array}
\end{equation*}
\caption{The first several Stieltjes constants.}
\label{Table:Stieltjes}
\end{table}

Although only the first few Stieltjes constants appear in our theorem statements, the proofs sometimes
involve bounds for all of the Stieltjes constants. An early result in this direction is due to Berndt,
who proved that 
\begin{equation*}
|\gamma_n| \leq \frac{(3+(-1)^n)n!}{n\pi^n}
\end{equation*}
for $n\geq 1$ \cite[Thm. 2]{Berndt} (1972); note that his normalization differs from that in \eqref{eq:Stieltjes}.
For our purposes we employ the bound
\begin{equation}\label{eq:ZhangWilliams}
 |\gamma_n| \leq \frac{(3+(-1)^n)(2n)!}{n^{n+1}(2\pi)^n}
\end{equation}
of Zhang and Williams \cite{ZhangWilliams}; see \cite{Blagouchine, Blagouchine2, Matsuoka1, Matsuoka2} 
for other estimates.

\begin{lem}\label{Lemma:ZhangWilliams}
For $n\geq 1$, we have $\displaystyle
 \frac{|\gamma_n|}{(n-1)!}
 \leq 4 \sqrt{2} \biggl(\frac{2}{\pi e} \biggr)^n$.
\end{lem}
\begin{proof}
We first show that
\begin{equation}\label{eq:Troublesome}
 \frac{(2n)!}{n!n^n} \leq \sqrt{2}\biggl(\frac{4}{e}\biggr)^n
\end{equation}
for all $n \geq 1$. Since
\begin{equation*}
\frac{(2n)!}{n!n^n}
= \frac{(n+1)(n+2)\cdots(2n)}{n^n}
= \prod_{k=1}^n \bigg(1+\frac{k}{n}\bigg),
\end{equation*}
it suffices to prove that
\begin{equation*}
\sum_{k=1}^n \ln\bigg(1+\frac{k}{n}\bigg)
\leq n\ln 4 - n + \frac{1}{2}\ln 2.
\end{equation*}
Since $f(x) = \ln(1+x)$ is concave on $[0,1]$, the trapezoidal rule underestimates the integral. Thus,
\begin{equation*}
\frac{1}{n}\Bigg(\frac{f(0)}{2} + \sum_{k=1}^{n-1} f\bigg(\frac{k}{n} \bigg) + \frac{f(1)}{2}\Bigg)
\leq \int_0^1 f(x)\dd x ,
\end{equation*}
and hence
\begin{equation*}
\sum_{k=1}^n f\bigg(\frac{k}{n} \bigg) \leq  n\int_0^1 f(x)\dd x + \frac{f(1)-f(0)}{2}.
\end{equation*}
This is equivalent to the desired estimate since $f(0) = 0$, $f(1) =\ln 2$, and $\int_0^1 f(x) \dd x = \ln 4 - 1$.
Now use \eqref{eq:ZhangWilliams} and \eqref{eq:Troublesome} to obtain
\begin{equation*}
\frac{|\gamma_n|}{(n-1)!}
\leq \frac{ 4(2n)!} {(n-1)!n^{n+1}(2\pi)^n}
= \frac{4}{(2\pi)^n}\cdot \frac{(2n)!}{n!n^n}
\leq 4\sqrt{2} \bigg(\frac{2}{\pi e}\bigg)^n. 
\qedhere
\end{equation*}
\end{proof}
\noindent %
The next lemma introduces several bounds that arise in what follows.
\begin{lem}\label{lem:basic_bounds}
Let $N\geq 2$ be an integer and define 
\begin{equation*}
R_N(s) = \sum_{n=N}^{\infty} \frac{(-1)^n}{n!}\gamma_n (s-1)^{n+1},
\end{equation*}
the tail of the Laurent expansion for $(s-1)\zeta(s)$. Suppose that $0\leq s-1 < e\pi/2 = 4.2698\ldots$. 
Set $w = 2(s-1)/(e\pi)$ so that $0\leq w < 1$. Then
\begin{align}
|R_N(s)|   &\leq \frac{4\sqrt{2}(s-1)}{N}\frac{w^N}{1-w},                                \label{eq:b_b_0}\\
|R'_N(s)|  &\leq 4\sqrt{2}\Big(1+\frac{1}{N}\Big)\frac{w^N}{1-w},                        \label{eq:b_b_1}\\
|R''_N(s)| &\leq \frac{8\sqrt{2}}{e\pi}w^{N-1}\Big(\frac{N}{1-w}+\frac{1}{(1-w)^2}\Big). \label{eq:b_b_2}
\end{align}
\end{lem}
\begin{proof}
Lemma \ref{Lemma:ZhangWilliams} ensures that
\begin{align*}
|R_N(s)|
&\leq \sum_{n=N}^{\infty} \frac{|\gamma_n|}{n!} (s-1)^{n+1} \\
&\leq  4\sqrt{2} (s-1)\sum_{n=N}^{\infty} \frac{1}{n}\bigg(\frac{2(s-1)}{e\pi}\bigg)^{n}
 \leq  \frac{4\sqrt{2}(s-1)}{N}\frac{w^N}{1-w}.
\end{align*}
In a similar way,
\begin{align*}
|R'_N(s)|
&\leq \sum_{n=N}^{\infty} \frac{(n+1)|\gamma_n|}{n!} (s-1)^{n} \\
&\leq  4\sqrt{2} \sum_{n=N}^{\infty}\frac{n+1}{n}w^n
 \leq  4\sqrt{2}\Big(1+\frac{1}{N}\Big)\frac{w^N}{1-w}.
\end{align*}
and
\begin{align*}
|R''_N(s)|
&\leq \sum_{n=N}^{\infty} \frac{|\gamma_n|}{(n-1)!}(n+1) (s-1)^{n-1} \\
&\leq  \frac{8\sqrt{2}}{e\pi}\sum_{n=N}^{\infty} (n+1)w^{n-1}
 =     \frac{8\sqrt{2}}{e\pi}w^{N-1}\cdot\frac{(N+1)(1-w)+w}{(1-w)^2} \\
&=     \frac{8\sqrt{2}}{e\pi}w^{N-1}\cdot\Big(\frac{N}{1-w}+\frac{1}{(1-w)^2}\Big).    
\qedhere
\end{align*}
\end{proof}

\subsection{Bernstein polynomials}\label{Subsection:Bernstein}
In several of our proofs, we must demonstrate that a certain polynomial is nonnegative over a given
interval. Instead of relying upon graphs or ad-hoc methods, we use a classical linear-algebraic approach
due to Sergei Bernstein.

For $k\in\{0,1,2,\ldots,n\}$, the $k$th \emph{Bernstein basis polynomial} of degree $n$ is
\begin{equation*}
 B_{k,n}(t) = \binom{n}{k}t^k(1-t)^{n-k}.
\end{equation*}
The polynomials $B_{0,n},B_{1,n},\ldots,B_{n,n}$ form a basis for the real vector space of polynomials of 
degree at most $n$. Thus, each polynomial $p$ of degree at most $n$ can be written uniquely in the form
\begin{equation*}
 p(t) = \sum_{k=0}^n \beta_k B_{k,n}(t).
\end{equation*}
If $p(t) = \sum_{j=0}^n\alpha_j t^j$, then one can compute the $\beta_k$'s by
\begin{equation}\label{eq:beta_k}
 \beta_k
 = \sum_{j=0}^k \alpha_j\frac{\binom{k}{j}}{\binom{n}{j}},
 \qquad
 0\leq k\leq n.
\end{equation}
%
Each Bernstein polynomial is nonnegative on $[0,1]$ and they sum to $1$. Thus, $p(t)\geq
\min(\beta_0,\beta_1,\ldots,\beta_n)$ for $t$ in $[0,1]$. In particular, $p(t)$ is nonnegative in case 
each \emph{Bernstein coefficient} $\beta_0$, $\beta_1$, $\ldots$, $\beta_n$ is nonnegative. Conversely,
any positive polynomial on $[0,1]$ can be expanded with respect to some Bernstein basis with $n \geq 
\deg p$ in such a way that each coefficient is nonnegative.  Unless stated otherwise, we take $n= 
\deg p$; there is only one case in this paper where $n > \deg p$ is necessary.
A linear transformation of the domain allows us to adapt our approach to every interval $[a,b]$. 
Finally, if the coefficients $\alpha_j$ and the endpoints $a$ and $b$ are known exactly, then the 
Bernstein coefficients of the transformed polynomial can likewise be computed exactly, so that their 
nonnegativity can be certified by exact arithmetic.
\begin{rk}\label{Remark:Resolve}
The \texttt{Mathematica} command \texttt{Resolve} provides another way to certify inequalities involving 
rational functions. Suppose that $p(t)$ and $q(t)$ are rational functions with exact coefficients and no 
poles in $[a,b]$. Then
\begin{center}
\texttt{Resolve[ForAll[t, a <= t <= b, p >= q], Reals]}
\end{center}
returns \texttt{True} if $p(t)\geq q(t)$ for all $t\in[a,b]$. Since \texttt{Resolve} uses exact 
quantifier-elimination methods, the resulting verification is not subject to numerical error. 
Stieltjes constants can be replaced by auxiliary variables constrained to rational intervals. 
For example, $\gamma=0.577\ldots$ may be replaced by a variable $g \in [\frac{57}{100}, \frac{58}{100}]$. 
One then applies \texttt{Resolve} to the resulting rational inequality, universally quantifying over $t$ 
and $g$. More accurate rational enclosures can be used if needed. In this way \texttt{Resolve} gives a 
rigorous certificate valid for every value of the constant in the specified interval. 
\end{rk}
Thus, we have two methods to verify inequalities involving rational functions.  In the following proofs
we have emphasized the Bernstein method since it is (in principle) humanly understandable and does not
rely upon a software black-box to return a yes-or-no answer with no further details.

\section{Proof of \texorpdfstring{$(s-1)\zeta(s) < U_{1,N}(s)$}{first-case}.}
\label{Section:JustZeta}%
Recall that
\begin{equation*}
 H_n - \ln n - \gamma < \frac{1}{2n}
\end{equation*}
for $n \geq 1$. For $n \geq 1$, let
\begin{equation*}
a_n = e^{H_{n-1}-\gamma}
\end{equation*}
and observe that the sequence is strictly increasing, tends to $+\infty$ and that
$a_{n+1} = a_n e^{1/n}$. For $n\geq 1$, we have
\begin{equation*}
 H_{n-1} - \gamma = H_n - \gamma - \frac{1}{n} < \ln n - \frac{1}{2n},
\end{equation*}
so that $a_n < ne^{-1/(2n)}$ upon exponentiation. Therefore,
\begin{equation*}
n^{1-s} e^{ \frac{s-1}{2n}} < a_n^{1-s}
\end{equation*}
since $s-1 > 0$. Since $\sinh(u/2) > (u/2)$ for $u >0$, we obtain
\begin{equation*}
ue^{-u/2} < 1-e^{-u},
\end{equation*}
which we apply to $u = (s-1)/n$ and get
\begin{equation*}
\frac{1}{n}e^{- \frac{s-1}{2n} } < \frac{1 - e^{-\frac{s-1}{n}}}{s-1}.
\end{equation*}
Next we have
\begin{align*}
 \int_{a_n}^{a_{n+1}} \frac{\dd x}{x^s}
 &= \frac{a_n^{1-s} - a_{n+1}^{1-s}}{s-1}
  = \frac{a_n^{1-s} (1 - e^{ - \frac{s-1}{n}} ) }{s-1} \\
 &> n^{1-s} e^{ \frac{s-1}{2n}} \cdot \frac{1}{n} e^{- \frac{s-1}{2n}}
  = \frac{1}{n^s}.
\end{align*}
Then, for any $N\geq 1$,
\begin{align*}
 \zeta(s)
 &= \sum_{n=1}^{\infty} \frac{1}{n^s}
  = \sum_{n=1}^{N-1} \frac{1}{n^s}
  + \sum_{n=N}^{\infty} \frac{1}{n^s}
  < \sum_{n=1}^{N-1} \frac{1}{n^s}
  + \int_{a_N}^{\infty} \frac{\dd x}{x^s} \\
 &= \sum_{n=1}^{N-1} \frac{1}{n^s}
  + \frac{a_N^{1-s}}{s-1}
  = \sum_{n=1}^{N-1} \frac{1}{n^s}
  + \frac{e^{ (\gamma-H_{N-1})(s-1)}}{s-1},
\end{align*}
which implies the desired result. 
\qed

\section{Proof of Theorem \ref{Theorem:FirstDerivative} (and, later, completion of Theorem \ref{Theorem:JustZeta})}
\label{Section:[(s-1)zeta(s)]'}
Recall that Theorem \ref{Theorem:FirstDerivative} concerns upper and lower bounds on $[(s-1)\zeta(s)]'$ for $s>1$.
As mentioned in the introduction, this proof relies upon Theorem \ref{Theorem:DoublePrime}, which is proved in
Section \ref{Section:DoublePrime}. This permits results about the first derivative to appear before the second.
We split the argument into two parts.

Suppose that $1 < s \leq 4$. Then Theorem \ref{Theorem:DoublePrime} says that
\begin{equation*}
\Big[\frac{-(1-\gamma)^2}{1-\gamma - 2\gamma_1(s-1)}\Big]'
< [(s-1)\zeta(s)]''
< \Big[\frac{\gamma - 1 - 2\gamma_1(s-1)}{1 + \tfrac{1}{8}(s-1)^2}\Big]'.
\end{equation*}
Integrate over $[1,s]$ and get
\begin{equation*}
1-\gamma +   \frac{-(1-\gamma)^2}{1-\gamma - 2\gamma_1(s-1)}
< -\gamma + [(s-1)\zeta(s)]'
< 1-\gamma + \frac{\gamma - 1 - 2\gamma_1(s-1)}{1 + \tfrac{1}{8}(s-1)^2},
\end{equation*}
which implies the desired inequality for $1 < s \leq 4$.

Now suppose that $s>4$. Then
\begin{equation*}
[(s-1)\zeta(s)]'
 = (s-1)\zeta'(s) + \zeta(s)
 = 1 + \sum_{n=2}^{\infty} \frac{1-(s-1)\ln n}{n^s},
\end{equation*}
so we must prove that
\begin{equation*}
\frac{1-\gamma + 2\gamma_1(s-1)}{1 + \frac{1}{8}(s-1)^2}
 < \sum_{n=2}^{\infty} \frac{(s-1)\ln n - 1}{n^s}
 < \frac{(1-\gamma)^2}{1-\gamma - 2\gamma_1(s-1)}.
\end{equation*}
For $s>4$, the function $t\mapsto (\ln t)/t^s$ is decreasing on $[2,\infty)$.
Moreover, $(s-1)\ln n-1>0$ for every $n\geq 3$.  Consequently,
\begin{align*}\sum_{n=2}^{\infty} \frac{(s-1)\ln n - 1}{n^s}
&\leq \frac{(s-1)\ln 2 - 1}{2^s} + (s-1)\sum_{n=3}^{\infty} \frac{\ln n}{n^s}       \\
&\leq \frac{(s-1)\ln 2 - 1}{2^s} + (s-1)\int_{2}^{\infty} \frac{\ln t}{t^s} \dd t   \\
&=    \frac{(s-1)\ln 2 - 1}{2^s} + \frac{2\ln 2 + 2/(s-1)}{2^s} \\
&=    \frac{s\ln 2 +\ln 2 - 1 + 2/(s-1)}{2^s},
\end{align*}
from which it follows that
\begin{equation*}
\frac{(s-1)\ln 2 - 1}{2^s}
 \leq \sum_{n=2}^{\infty} \frac{(s-1)\ln n - 1}{n^s}
 \leq \frac{s\ln 2 +\ln 2 - 1 + 2/(s-1)}{2^s}.
\end{equation*}
For $s>4$, we have
\begin{equation*}
 1 - \gamma + 2\gamma_1(s-1) < 1 - \gamma + 6 \gamma_1 < 0
\end{equation*}
and
\begin{equation*}
 (s-1)\ln 2- 1 > 3\ln 2 - 1 > 1,
\end{equation*}
so
\begin{equation*}
 \frac{1-\gamma + 2\gamma_1(s-1)}{1 + \frac{1}{8}(s-1)^2} < \frac{(s-1)\ln 2 - 1}{2^s}.
\end{equation*}
This proves the desired lower bound.

The required upper bound will follow if we establish that
\begin{equation*}
\underbrace{(s-1)\bigg(\frac{s\ln 2 +\ln 2 - 1 + 2/(s-1)}{2^s}\bigg)}_{L(s)}
<  \underbrace{(s-1)\bigg(\frac{(1-\gamma)^2}{1-\gamma - 2\gamma_1(s-1)}\bigg)}_{R(s)}
\end{equation*}
for $s>4$.  Since $(s+1)\ln 2 > 5 \ln 2 > 3$ there and $1-\gamma > 0$, we have
\begin{align*}
 L(4)  &=  \frac{3}{16} \bigg(5 \ln 2-\frac{1}{3}\bigg) = 0.5873\ldots, \\[2pt]
 L'(s) &= -\frac{1 + (s-1) ((s+1)\ln 2 - 3)\ln 2}{2^s} < 0,\\[2pt]
 R(4)  &=  \frac{3 (1-\gamma )^2}{ 1 -\gamma -6 \gamma_1} = 0.6237\ldots, \quad \text{and}\\[2pt]
 R'(s) &=  \frac{(1-\gamma)^3}{(1 - \gamma - 2 \gamma_1 (s-1))^2} > 0.
\end{align*}
Thus, $L(s) < R(s)$ on $s>4$. 
\qed
\medskip

The bounds $L_2(s)<(s-1)\zeta(s)<U_2(s)$ in Theorem \ref{Theorem:JustZeta} immediately follow by
integrating the bounds in Theorem \ref{Theorem:FirstDerivative}.
%
%

\section{Proof of Theorem \ref{Theorem:DoublePrime}}\label{Section:DoublePrime}
Recall that Theorem \ref{Theorem:DoublePrime} concerns upper and lower bounds on $[(s-1)\zeta(s)]''$ for
$1 < s \leq 4$. The definition \eqref{eq:Stieltjes} of the Stieltjes constants $\gamma_n$ ensures that
\begin{equation*}
[(s-1)\zeta(s)]'' = \sum_{n=1}^{\infty} \frac{(-1)^n}{n!}\gamma_n (n+1)n(s-1)^{n-1}.
\end{equation*}
The bound \eqref{eq:b_b_2} in Lemma \ref{lem:basic_bounds} yields%
\begin{align*}
\Big|[(s - 1)\zeta(s)]'' & - \sum_{n=1}^{N-1} \frac{(-1)^n}{n!}\gamma_n n(n + 1)(s - 1)^{n-1}\Big|\\
&= |R''_N(s)|
 \leq \frac{8\sqrt{2}}{e\pi}w^{N-1}\Bigl(\frac{N}{1-w}+\frac{1}{(1-w)^2}\Bigr),
\intertext{which simplifies to}
&\leq \Big(\frac{2(s-1)}{e\pi}\Big)^{N-1}(5N {+} 15)
\end{align*}
for every $N \geq 2$ since $w= 2(s-1)/(e\pi) \leq 6/(e\pi)$ for $s\in[1,4]$.
Fix $N=30$ and observe that Theorem \ref{Theorem:DoublePrime} will follow if we can establish that
\begin{equation*}
 \frac{-2\gamma_1(1-\gamma)^2}{(1-\gamma - 2\gamma_1(s-1))^2}
 < \sum_{n=1}^{29} \frac{(-1)^n}{n!}\gamma_n n(n +1)(s -1)^{n-1}
     - 165\Big(\frac{2(s-1)}{e\pi}\Big)^{29}
\end{equation*}
and
\begin{align*}
 & \sum_{n=1}^{29} \frac{(-1)^n}{n!}\gamma_n n(n +1)(s -1)^{n-1} + 165\Big(\frac{2(s-1)}{e\pi}\Big)^{29} \\
 &\qquad\qquad
  < -\frac{2\gamma_1-\frac{1}{4}(1 - \gamma)(s-1) - \frac{1}{4}\gamma_1 (s-1)^2}{(1 + \frac{1}{8}(s-1)^2)^2}.
\end{align*}
For $1<s\leq 4$, subtract the two
sides of the respective inequalities and multiply by
\begin{equation*}
\frac{(1-\gamma-2\gamma_1(s-1))^2}{s-1}
\qquad\text{and}\qquad
\frac{(1+(s-1)^2/8)^2}{s-1},
\end{equation*}
respectively.  In each case the numerator vanishes at $s=1$, so after cancellation the resulting
expressions extend to polynomials $p_-(s)$ and $p_+(s)$ on $[1,4]$, of degrees $30$ and $32$,
respectively. The corresponding Bernstein coefficients (see Subsection \ref{Subsection:Bernstein}) of
$p_-(1+3t)$ and $p_+(1+3t)$ are larger than, resp., $0.00399$ and $0.00439$.
%
%
Since each Bernstein coefficient is strictly positive, $p_-(s)$ and $p_+(s)$ are strictly positive on
$[1,4]$.
\qed

\section{Proof of Theorem \ref{Theorem:Combo}}\label{Section:Combo}
We seek explicit upper and lower bounds on $(\zeta'/\zeta)(s) + (s-1)^{-1}$ for $s>1$. Since the proofs
of these bounds are substantially different in nature, we present them in separate subsections.

\subsection{Proof of the lower bound}
Let $F(s)=(s-1)\zeta(s)$.
We give a rational upper bound for the function $U_{1,2}(s)$ of Theorem \ref{Theorem:JustZeta}. For all $s>1$, 
we have
\begin{equation}\label{eq:upper-simplified}
 U_{1,2}(s)
 = s-1 + e^{(\gamma-1)(s-1)}
 \leq s - (1-\gamma)\frac{s-1}{s} - \frac{(s-1)^2}{3 s^2}.
\end{equation}
After multiplication by $3s^2$, this is equivalent to
\begin{equation*}
 f(s)=3s^2e^{(\gamma-1)(s-1)} - 3s^2 - 3(\gamma-1)s(s-1)+(s-1)^2\leq 0.
\end{equation*}
The function $f$ has a zero of order $2$ at $s=1$, and we have 
\begin{align*}
 f''(s)
 &= 3((\gamma-1)^2s^2+4(\gamma-1)s+2)e^{(\gamma-1)(s-1)}-6\gamma+2, \\
 f'''(s)
 &= 3(\gamma-1)((\gamma-1)^2s^2+6(\gamma-1)s+6)e^{(\gamma-1)(s-1)}.
\end{align*}
The function $f'''$ is negative at $s=1$. Moreover, its zeros are at $s=s_\pm$, where $s_\pm = 
\frac{3\pm\sqrt{3}}{1-\gamma}$. Both are larger than $1$, $f''(1)=3\gamma^2-1 < 0$,  and 
\begin{equation*}
f''(s_+)=6(1+\sqrt{3})e^{-2-\sqrt{3}-\gamma}-6\gamma+2= -1.24....
\end{equation*}
Therefore, for $s>1$, $f''(s)<0$ and so $f(s)<0$ as well, which proves \eqref{eq:upper-simplified}.
Since $\gamma_1<0$, we have
\begin{equation*}
1-\frac{(1-\gamma)^2}{1-\gamma-2\gamma_1(s-1)}
> 1-\frac{(1-\gamma)^2}{1-\gamma} =\gamma>0.
\end{equation*}
Thus, the lower bound for $F'(s)$ in Theorem \ref{Theorem:FirstDerivative}, the strict upper bound
$F(s)<U_{1,2}(s)$ in Theorem \ref{Theorem:JustZeta}
and \eqref{eq:upper-simplified} yield
\begin{align*}
\frac{\zeta'}{\zeta}(s) &+ \frac{1}{s-1}
= \frac{ [(s-1)\zeta(s)]'}{(s-1)\zeta(s)} \\
&> \Big[1 - \frac{(1-\gamma)^2}{1-\gamma - 2\gamma_1(s-1)}\Big]
     \times
     \frac{1}{s - (1-\gamma)\frac{s-1}{s} - \frac{(s-1)^2}{3 s^2}}.
\end{align*}
When simplified, this is a rational function $L(t)$ in $t = s-1$ with numerator
\begin{equation*}
6 \gamma_1 t^3
 + 3 (\gamma^2- \gamma+4 \gamma_1) t^2
 + 6 (\gamma^2-\gamma+\gamma_1) t
 + 3 (\gamma^2- \gamma )
\end{equation*}
and denominator
\begin{align*}
&6\gamma_1 t^4
 + (3\gamma+10\gamma_1+6\gamma\gamma_1-3)t^3
 + (3\gamma^2+2\gamma+12\gamma_1+6\gamma\gamma_1-5)t^2 \\
&\qquad +(3\gamma^2+3\gamma+6\gamma_1-6)t
 + 3\gamma-3.
\end{align*}
This provides an explicit lower bound, albeit one that is cumbersome. We seek an alternative rational
function $P(t)$ of smaller degree with
\begin{equation*}
 P(0) = L(0), 
 \quad
 P'(0) = L'(0), 
 \quad \text{and} \quad
 \lim_{t\to\infty} tP(t) = \lim_{t\to\infty} tL(t) = 1.
\end{equation*}
This leads us to consider
\begin{equation*}
P(t) = \frac{\gamma + c t}{1+ 2 t + c t^2},
\end{equation*}
in which $c = 2\gamma - 2\gamma_1 - \gamma^2$. 
A direct simplification gives
\begin{equation*}
L(t)-P(t)
= \frac{t^2(a_2+a_3t+a_4t^2)}
       {(1-\gamma-2\gamma_1t)(1+2t+ct^2)\bigl(3t^3+(3\gamma+5)t^2+(3\gamma+6)t+3\bigr)},
\end{equation*}
in which the exact coefficients $a_i$ are too cumbersome to display, but
\begin{equation*}
a_2 > 0.2403,\qquad
a_3 > 0.3484,\qquad
a_4 > 0.0001.
\end{equation*}
Since $c>0$ and $\gamma_1<0$, every factor in the denominator is
positive for $t>0$.  Hence $L(t)>P(t)$ for $t>0$.  Thus,
\begin{equation*}
 P(s-1) < \frac{\zeta'}{\zeta}(s) + \frac{1}{s-1},
\end{equation*}
which proves the desired lower bound.

\subsection{Proof of the upper bound}
Suppose that $s\geq 4$.  First note that
\begin{equation}\label{eq:TwoGammas}
 b = \gamma + \frac{2 \gamma_1}{\gamma} 
   = 0.3249\ldots 
   > \frac{8}{25}
 \quad \text{and} \quad
 \frac{2 \gamma_1}{\gamma} 
   = -0.2523 \ldots 
   < -\frac{1}{4}.
\end{equation}
Then take the logarithmic derivative of the Euler product and get
\begin{equation*}
 - \frac{\zeta'}{\zeta}(s) = \sum_p \frac{\ln p}{p^s-1} > \frac{\ln 2}{2^s}.
\end{equation*}
Consequently,
\begin{equation*}
 \frac{\zeta'}{\zeta}(s) + \frac{1}{s-1}
 < \frac{1}{s-1} - \frac{\ln 2}{2^{s}}
\end{equation*}
hence to prove the upper bound it suffices to show that
\begin{equation*}
\frac{1}{s-1} - \frac{\ln 2}{2^{s}} < \frac{\gamma}{1+(\gamma + 2 \gamma_1/\gamma)(s-1)},
\end{equation*}
or, equivalently,
\begin{equation*}
 2^s \bigg(1 + \frac{ 2 \gamma_1}{\gamma}(s-1)\bigg)
 < (\ln 2)(s-1)\big[1 + b (s-1) \big].
\end{equation*}
From \eqref{eq:TwoGammas}, we have
\begin{equation*}
 1 + \frac{2\gamma_1}{\gamma}(s-1) 
 < 1 - \frac{1}{4}(s-1) 
 = \frac{1}{4}(5-s)
\end{equation*}
and
\begin{equation*}
 1 + b(s-1) 
 > 1 + \frac{8}{25}(s-1).
\end{equation*}
Thus, it suffices to prove that $f(s) < g(s)$ for $s \geq 4$, in which 
\begin{equation*}
 f(s) = 2^s (5-s) 
 \quad \text{and} \quad
 g(s) = (4 \ln 2)(s-1)\bigg(1 + \frac{8}{25}(s-1)\bigg).
\end{equation*}
Observe that $f$ is decreasing since 
\begin{equation*}
f'(s) 
 =   -2^s \big(1 + (s-5)\ln 2 \big) 
\leq -2^s(1 - \ln 2)<0    
\end{equation*}
for $s \geq 4$ and $g$ is increasing there.  Since $f(4) = 16$ and $g(4) = \frac{588}{25}\ln 2 
= 16.3028 \ldots$, the desired upper bound follows for $s \geq 4$.
Suppose that $s\in[1,4]$ and let $b$ remain as defined in \eqref{eq:TwoGammas}.
Recall $F(s) = (s-1)\zeta(s)$, in which the singularity at $s=1$ is removable. Since
\begin{equation*}
\frac{F'(s)}{F(s)} 
 = \frac{\zeta'}{\zeta}(s)+\frac{1}{s-1},    
\end{equation*}
the desired inequality is equivalent to
\begin{equation}\label{eq:Qpositive}
 Q(s)
  = \gamma F(s)-\big(1+b(s-1)\big)F'(s) 
  > 0
\end{equation}
for $1 < s \leq 4$.
The Laurent expansion \eqref{eq:Stieltjes} of $\zeta(s)$ at $s=1$ reveals that 
\begin{equation*}
F(1)=1, 
\qquad 
F'(1)=\gamma, 
\qquad\text{and}\qquad 
F''(1)=-2\gamma_1,
\end{equation*}
so $Q(1)=Q'(1)=0$.  Consequently, $Q$ has a zero of order at least two at $s=1$.
For $N\geq 2$, define
\begin{equation*}
F_N(s) = 1+\sum_{n=0}^{N-1}\frac{(-1)^n\gamma_n}{n!}(s-1)^{n+1}
\end{equation*}
and
\begin{align*}
 Q_N(s) = \gamma F_N(s) - \big(1+b(s-1)\big)F_N'(s).
\end{align*}
The polynomial $F_N$ satisfies $F_N(1) = 1$, $F_N'(1) = \gamma$, and $F_N''(1) =-2\gamma_1$, from which 
it follows that $Q_N(1)=Q_N'(1)=0$. In particular, the polynomial $Q_N$ is divisible by $(s-1)^2$. Write
\begin{equation*}
P_N(s) = \frac{Q_N(s)}{(s-1)^2},    
\end{equation*}
which is a polynomial of degree at most $N-2$.  For $s>1$, we have
\begin{equation*}
 \frac{Q(s)}{(s-1)^2} 
 = P_N(s) + \frac{\gamma R_N(s)-(1+b(s-1))R'_N(s)}{(s-1)^2}.
\end{equation*}
For $N=18$, the smallest Bernstein coefficient of the degree-$16$ polynomial $P_{18}(1+3t)$ is greater 
than $0.0092$, so $P_{18}(s) > 0.0092$ for $s \in [1,4]$. 
%
%
Lemma \ref{lem:basic_bounds} ensures that
\begin{align*}
 &\bigg|\frac{\gamma R_{18}(s)-(1+b(s-1))R'_{18}(s)}{(s-1)^2}\bigg| \\
 &\qquad 
  < \frac{4 \sqrt{2} (\frac{2}{\pi e})^{18}}{1 - \frac{2(s-1)}{\pi e}}
    \bigg(\frac{\gamma}{18}(s-1)^{17} + \frac{19}{18}\big(1 + b(s-1)\big)(s-1)^{16} \bigg).    
\end{align*}
This upper bound is increasing on $1 < s \leq 4$ and equals $0.0080223\ldots < 0.0081$ at $s=4$. 
%
%
Therefore,
\begin{equation*}
 \frac{Q(s)}{(s-1)^2} 
 > 0.0092 - 0.0081 
 > 0,
\end{equation*}
so \eqref{eq:Qpositive} holds.  This establishes the desired upper bound for $1 < s \leq 4$
and completes the proof. 
\qed

\section{Proof of Theorem \ref{Theorem:Decreasing}}\label{Section:Decreasing}
We must show that $(\zeta'/\zeta)(s) + (s-1)^{-1}$ is strictly decreasing for $s>1$. We first consider 
$s\geq 4$, which can be handled via a simple Dirichlet-series argument, and then address the more 
delicate region $1 < s < 4$ with a certain sum over the nontrivial zeros of the zeta function.

For $s\geq 4$, we use $(\zeta'/\zeta)(s) = - \sum_{n=1}^{\infty} \Lambda(n) n^{-s}$, in which 
$\Lambda(n)$ is the von Mangoldt function. Then
\begin{align*}
\Bigl(\frac{\zeta'}{\zeta}\Bigr)'(s)
&  =  \sum_{n=1}^{\infty}\frac{\Lambda(n)\ln n}{n^s}
 \leq \sum_{n=1}^{\infty}\frac{\ln^2 n}{n^s} \\
&\leq \frac{4}{(es)^2}
     +  \int_1^{\infty} \frac{\ln^2 x}{x^s}\dd x
   =  \frac{4}{(es)^2}
     +  \frac{2}{(s-1)^3}
 \leq \frac{1}{(s-1)^2}
\end{align*}
since $x\mapsto\frac{\ln^2 x}{x^s}$ attains its maximum $4/(es)^2$ at $x=\exp(2/s)$ and because $s\geq 4$.
%
%
Thus, $(\zeta'/\zeta)(s) + (s-1)^{-1}$ is decreasing on $[4,\infty)$.

Suppose that $1<s<4$. Hadamard's factorization provides
\begin{equation*}
s(s-1)\Gamma\Bigl(\frac{s}{2}\Bigr)\zeta(s)
= e^{a+bs}\prod_{\rho}e^{s/\rho}\Bigl(1-\frac{s}{\rho}\Bigr)
\end{equation*}
for certain constants $a$, $b \in \CM$, in which $\rho$ runs over the nontrivial zeros of $\zeta$. Since
$s\Gamma(s)=\Gamma(s+1)$, a logarithmic differentiation of the above shows that
\begin{equation*}
\frac{\zeta'}{\zeta}(s)
  + \frac{1}{s-1}
= b - \frac{1}{2}\frac{\Gamma'}{\Gamma}\Bigl(\frac{s}{2}+1\Bigr)
  + \sum_{\rho}\Bigl(\frac{1}{s-\rho} + \frac{1}{\rho}\Bigr).
\end{equation*}
Take another derivative.  If $\rho=\beta+i\tau$ is a nontrivial zero, then pairing $\rho$ with its
conjugate gives
\begin{equation*}
-\frac{1}{(s-\rho)^2}-\frac{1}{(s-\overline{\rho})^2}
= 2\,\frac{\tau^2-(s-\beta)^2}{ ((s-\beta)^2+\tau^2 )^2}.
\end{equation*}
Consequently,
\begin{align}
\bigg(\frac{\zeta'}{\zeta}(s) + \frac{1}{s-1}\bigg)'
&= - \frac{1}{4}\bigg(\frac{\Gamma'}{\Gamma}\bigg)'\bigg(\frac{s}{2}+1\bigg)  - \sum_{\rho}\frac{1}{(s-\rho)^2} \label{eq:SumRho}\\
&= - \sum_{k=1}^{\infty} \frac{1}{(s+2k)^2}
   + \sum_{\rho}\frac{\tau^2 - (s-\beta)^2}{((s-\beta)^2 + \tau^2)^2} \nonumber
\end{align}
since \cite[Thm. 1.2.5 \& (1.2.14)]{AnAsRoy} ensures that
\begin{equation*}
\bigg(\frac{\Gamma'}{\Gamma}\bigg)'(s)
= \sum_{k=0}^{\infty} \frac{1}{(s+k)^2}    .
\end{equation*}
Here and below the sums over $\rho$ run over all nontrivial zeros, counted by multiplicity. Thus, the
contribution of each conjugate pair occurs twice in the second sum. For $1 < s < 4$, the first series
satisfies
\begin{equation*}
-\sum_{k=1}^{\infty} \frac{1}{(s+2k)^2}
\leq -\sum_{k=1}^{\infty} \frac{1}{(4+2k)^2}
=    \frac{15-2 \pi ^2}{48}
<    -0.09.
\end{equation*}
We next bound the second series.  Write $\rho=\beta+i\tau$ and set
\begin{equation*}
G(u) = \frac{\tau^2-u}{(u+\tau^2)^2},
\quad \text{so that}\quad
G'(u) = \frac{u-3\tau^2}{(u+\tau^2)^3}.
\end{equation*}
Since $0<\beta<1$ and $1<s<4$, we have $0<(s-\beta)^2<16$. The imaginary part of every nontrivial zero 
has absolute value at least $14.13\ldots$, so
\begin{equation*}
(s-\beta)^2 < 3\tau^2
\end{equation*}
and hence $G$ is decreasing throughout the relevant range. Since $(s-\beta)^2>(1-\beta)^2$ for $s>1$,
each summand of the second series is at most equal to its value at $s=1$. Thus, the second series is 
less than
\begin{equation}\label{eq:sum-rho}
\sum_{\rho}\frac{\tau^2 - (1-\beta)^2}{((1-\beta)^2 + \tau^2)^2}
= -\sum_{\rho} \frac{1}{(1-\rho)^2}.
\end{equation}
The value of this sum can be explicitly recovered evaluating \eqref{eq:SumRho} at $s=1$, and getting
that
\begin{align*}
-\gamma^2-2\gamma_1
&= \bigg(\frac{\zeta'}{\zeta}(s) + \frac{1}{s-1}\bigg)' \bigg|_{s=1}
 = -\sum_{k=1}^{\infty} \frac{1}{(1+2k)^2}- \sum_{\rho} \frac{1}{(1-\rho)^2} \\
&= -\bigg(\frac{\pi^2}{8}-1\bigg) - \sum_{\rho} \frac{1}{(1-\rho)^2},
\end{align*}
so
\begin{equation*}
-\sum_{\rho} \frac{1}{(1-\rho)^2}
 = \frac{\pi ^2}{8}-2 \gamma _1-\gamma ^2-1
 < 0.05.
\end{equation*}
Combining this with the estimate for the first series gives
\begin{equation*}
\bigg(\frac{\zeta'}{\zeta}(s)+\frac{1}{s-1}\bigg)'
\leq - 0.09 + 0.05 <0.
\end{equation*}
Thus, $(\zeta'/\zeta)(s) + (s-1)^{-1}$ decreases for $s\in(1,4)$, and hence for all $s>1$. 
\qed

\section{Proof of Theorem \ref{Theorem:ZetaPrime}}\label{Section:ZetaPrime}
Let $N$ be a positive integer and define
\begin{equation*}
Z(s) = \zeta'(s) + \frac{1}{(s-1)^2}
\quad \text{and} \quad
Z_N(s) = \sum_{n=1}^{N-1} \frac{(-1)^n \gamma_n}{(n-1)!}(s-1)^{n-1}.
\end{equation*}
For $0 < s-1 < \pi e/2$, \eqref{eq:Stieltjes} and Lemma \ref{Lemma:ZhangWilliams} ensure that
\begin{equation}\label{eq:ChainZEN}
Z_N(s) - E_N(s) \leq Z(s) \leq Z_N(s) + E_N(s),
\end{equation}
in which
\begin{equation*}
E_N(s) = \frac{8\sqrt{2}}{\pi e}\cdot \frac{w^{N-1}}{1 - w}
\quad
\text{with}
\quad
w= \frac{2(s-1)}{\pi e}.
\end{equation*}
Suppose that $s\in (1,4]$, so that $E_N(s) < 5 w^{N-1}$.
To establish the desired inequalities in this range, it suffices to prove that
\begin{equation}\label{eq:lower}
\frac{\alpha}{1 + \beta (s-1) + \alpha (s-1)^2 } \leq Z_N(s) - 5 w^{N-1}
\end{equation}
and 
\begin{equation}\label{eq:upper}
Z_N(s) + 5 w^{N-1} < \frac{\alpha}{1+ \beta(s-1) +\delta(s-1)^2}.
\end{equation}
Both inequalities involve rational expressions in $(s-1)$, so they can be addressed with the 
Bernstein-polynomial method. We chose $\alpha$ and $\beta$ so multiplication of \eqref{eq:lower} 
by the denominator on the left side yields a polynomial divisible by $(s-1)^2$ when $N\geq 3$. 

Fix $N = 26$. Let $p_-(s)$ denote the degree-$25$ polynomial obtained after division. The smallest 
Bernstein coefficient of $p_-(1+3t)$ is larger than $10^{-3}$, so \eqref{eq:lower} holds. We chose 
$\alpha$, $\beta$, and $\delta$ so multiplication of \eqref{eq:upper} by the denominator on the right 
side yields a polynomial divisible by $(s-1)^3$ when $N\geq 4$. Let $p_+(s)$ denote the degree-$24$ 
polynomial that is obtained after the division. Each Bernstein coefficient of $p_+(1+3t)$ is positive 
(the smallest is greater than $10^{-5}$), so \eqref{eq:upper} holds.

Suppose that $s\geq 4$. Then
\begin{equation*}
\frac{2^s}{\ln 2}\sum_{n=3}^{\infty} \frac{\ln n}{n^s} 
= \frac{1}{\ln 2}\sum_{n=3}^{\infty} \bigg(\frac{2}{n}\bigg)^{s}\ln n    
\end{equation*}
is strictly decreasing since each summand is strictly decreasing. Its value at $s=4$ is strictly less 
than $3/5$, so
\begin{equation*}
- \zeta'(s) = \sum_{n=2}^{\infty} \frac{\ln n}{n^s} <  \frac{8}{5}\cdot \frac{\ln 2}{2^s}
\end{equation*}
for $s \geq 4$. To establish the desired lower bound, it suffices to prove that
\begin{equation*}
\frac{\alpha}{1 + \beta (s-1) + \alpha (s-1)^2}
\leq \frac{1}{(s-1)^2} - \frac{8}{5}\cdot \frac{\ln 2}{2^s}.
\end{equation*}
Rewrite this as
\begin{equation}\label{eq:PlugIntoThis}
 0 
 \leq 2^{s-4}\left(\frac{1 + \beta  (s-1)}{(s-1)^2 (1 + \beta(s-1)+\alpha(s-1)^2 )}\right)
     - \frac{\ln 2}{10}.
\end{equation}
Since the rational factor above is positive and
\begin{equation*}
 2^{s-4} 
 \geq 1 + (s-4)\ln 2 + \frac{ (s-4)^2 (\ln 2)^2}{2} + \frac{ (s-4)^3 (\ln 2)^3}{6},
\end{equation*}
it suffices to verify the nonnegativity, for $s\geq 4$, of the polynomial obtained after substituting this 
lower bound into \eqref{eq:PlugIntoThis} and clearing the positive denominator.  In this manner, one 
obtains a quartic $q_-(s)$ with explicit coefficients whose strict positivity on $s \geq 4$ can be 
verified via basic calculus.  Alternately, one can substitute $s = \frac{3t-4}{t-1}$ with $t \in [0,1)$ 
into $q_-(s)$, multiply by $(t-1)^4$ to clear the denominator, and then certify positivity by expanding 
the resulting degree-$4$ polynomial in the degree-$13$ Bernstein basis.

We now consider the desired upper bound of the theorem.  Suppose still that $s \geq 4$.
Since $\zeta'(s) < 0$, the bound holds for $s \geq 8$ since
\begin{equation*}
\frac{1}{(s-1)^2} < \frac{\alpha}{1+ \beta(s-1) + \delta(s-1)^2}
\end{equation*}
is equivalent to the positivity of $(\alpha - \delta)(s-1)^2 - \beta(s-1) - 1$, a quadratic with positive 
leading coefficient whose largest zero is $\approx 7.821$. It remains to consider $4 \leq s \leq 8$.
Since $0\leq (s-4)\ln n\leq 4\ln 10<10$ for $2\leq n\leq 10$ and
\begin{equation*}
e^{-x}\geq\bigg(1-\frac{x}{10}\bigg)^{10}
\end{equation*}
for $0 \leq x < 10$, we have
\begin{align*}
\zeta'(s)+\frac{1}{(s-1)^2}
&=    \frac{1}{(s-1)^2} -\sum_{n=2}^{\infty}\frac{\ln n}{n^s}
 <    \frac{1}{(s-1)^2} -\sum_{n=2}^{10}\frac{\ln n}{n^s}                 \\
&=    \frac{1}{(s-1)^2} -\sum_{n=2}^{10}\frac{\ln n}{n^4} e^{-(s-4)\ln n} \\
&\leq \frac{1}{(s-1)^2} -\sum_{n=2}^{10}\frac{\ln n}{n^4} \left(1-\frac{(s-4)\ln n}{10}\right)^{10}\\
&\leq \frac{\alpha}{1+\beta(s-1)+\delta(s-1)^2}.
\end{align*}
The final inequality follows from a now-standard argument: express the inequality as the nonnegativity of
a certain rational function, multiply by the positive denominator $(s-1)^2(1+\beta(s-1)+\delta(s-1)^2)$,
and obtain a degree-$14$ polynomial $q_+(s)$. The Bernstein coefficients of $q_+(4+4t)$ with $t \in
[0,1]$ are all positive; the smallest is greater than $0.003$.

This concludes the proof of Theorem \ref{Theorem:ZetaPrime}. 
\qed

\end{document}